\documentclass{amsart}
\usepackage{amsfonts}
\usepackage{graphicx}
\usepackage{tabularx}
\usepackage{array}
\usepackage[usenames,dvipsnames]{color}
\usepackage{comment}
\usepackage{amsmath}
\usepackage{amsthm}
\usepackage{amssymb}
\usepackage{fullpage}
\usepackage{xcolor}
\usepackage{tikz}
\usetikzlibrary{graphs}
\usetikzlibrary{graphs.standard}

\usepackage{listings}
\usetikzlibrary{arrows}

\tikzstyle{legend_general}=[rectangle, rounded corners, thin,
                          top color= white,bottom color=lavander!25,
                          minimum width=2.5cm, minimum height=0.8cm,
                          violet]

\DeclareMathOperator{\col}{col}
\DeclareMathOperator{\Bob}{Bob}

\newtheorem{theorem}{Theorem}[section]

\newtheorem{question}[theorem]{Question}

\newtheorem{corollary}[theorem]{Corollary}
\theoremstyle{definition}

\title{Graph colorings with restricted bicolored subgraphs: II. The graph coloring game}
\author{Peter Bradshaw}
\address{Department of Mathematics, Simon Fraser University, Burnaby, BC, Canada}
\email{pabradsh@sfu.ca}

\thanks{The author of this work has been partially supported by a supervisor's grant from the Natural Sciences and Engineering Research Council of Canada (NSERC)}
\begin{document}
\maketitle
\begin{abstract}
We consider the \emph{graph coloring game}, a game in which two players take turns properly coloring the vertices of a graph, with one player attempting to complete a proper coloring, and the other player attempting to prevent a proper coloring. We show that if a graph $G$ has a proper coloring in which the \emph{game coloring number} of each bicolored subgraph is bounded, then the \emph{game chromatic number} of $G$ is bounded. As a corollary to this result, we show that for two graphs $G_1$ and $G_2$ with bounded game coloring number, the Cartesian product $G_1 \square G_2$ has bounded game chromatic number, answering a question of X. Zhu. We also obtain an upper bound on the game chromatic number of the strong product $G_1 \boxtimes G_2$ of two graphs.
\end{abstract}

\section{Introduction}
 The \emph{graph coloring game} is a game played on a finite graph $G$ with perfect information by two players, Alice and Bob. In the graph coloring game, Alice moves first, and Alice and Bob take turns coloring vertices of $G$. On each player's turn, the player chooses an uncolored vertex $v \in V(G)$ and colors $v$ using a color from a predetermined set $S$. Each player must color $G$ properly on each turn; that is, a player may not color a vertex $v$ with a color that appears in the neighborhood of $v$. Alice wins the game if each vertex of $G$ is properly colored, and Bob wins the game if every color of $S$ appears in the neighborhood of some uncolored vertex $v$, as this means that $v$ can never be properly colored. The \emph{game chromatic number} of $G$, written $\chi_g(G)$, is the minimum integer $k$ for which Alice has a winning strategy in the graph coloring game on $G$ when playing with a color set $S$ of $k$ colors. 

The game chromatic number was introduced by Bodlaender in \cite{Bodlaender} in 1990 and has received considerable attention since its invention. It is straightforward to show that for a graph $G$ of chromatic number $\chi(G)$ and maximum degree $\Delta(G)$, the following inequality holds:
$$\chi(G) \leq \chi_g(G) \leq \Delta(G) + 1.$$
The upper bound of $\Delta(G)+1$ is far from optimal in many cases, however. For instance, when $G$ is a forest, $\chi_g(G) \leq 4$ \cite{Faigle}, and when $G$ is planar, $\chi_g(G) \leq 17$ \cite{ZhuRefined}. Furthermore, $\chi_g(G)$ can be bounded above by other parameters of $G$. For instance, when $G$ has treewidth at most $w$, $\chi_g(G) \leq 3w + 2$ \cite{ZhuKtrees}, and when $G$ has genus at most $g$, $\chi_g(G) \leq \lfloor \frac{1}{2}(3 \sqrt{1 + 48g} + 23) \rfloor$ \cite{ZhuKtrees}. Furthermore, Dinski and Zhu \cite{DinskiZhu} show that $\chi_g(G)$ is bounded above by a function of the \emph{acyclic chromatic number} of $G$, written $\chi_a(G)$, which is the minimum number of colors needed to give $G$ a proper coloring in which every bicolored subgraph of $G$ is a forest. Dinski and Zhu give the following upper bound:
$$\chi_g(G) \leq \chi_a(G) ( \chi_a(G) + 1).$$

Similar to the graph coloring game, the \emph{graph marking game} is also a game played on a finite graph $G$ with perfect information by two players, Alice and Bob. In the graph marking game, first considered by Faigle et al.~ in \cite{Faigle}, the players take turns, with Alice moving first, and on a player's turn, the player chooses an unmarked vertex $v \in V(G)$ and marks $v$ with a black pen. The game ends when all vertices in $G$ have been marked. After a play of the graph marking game, each vertex $v \in V(G)$ receives a \emph{score} equal to the number of neighbors of $v$ that were already marked at the time that $v$ was marked. A play of the graph marking game on $G$ is then given a score equal to the maximum score over all vertices of $V(G)$, plus one. Alice's goal in the graph marking game is to minimize the score of the play, and Bob's goal is to maximize the score of the play. The \emph{game coloring number} of $G$, written $\col_g(G)$, is the minimum integer $t$ for which Alice has a strategy to limit the score of a play on $G$ to $t$. It is straightforward to show that $\chi_g(G) \leq \col_g(G)$. 

When attempting to find an upper bound for the game chromatic number of a graph $G$, it is often convenient to consider the graph marking game on $G$ and find an upper bound for $\col_g(G)$. The reason for this is that the game coloring number satisfies certain convenient properties that are not satisfied by the game chromatic number. For instance, when $H$ is a subgraph of $G$, Wu and Zhu show that $\col_g(H) \leq \col_g(G)$ \cite{WuLower}. On the other hand, Tuza and Zhu show that the ``cocktail party graph," obtained from the complete bipartite graph $K_{n,n}$ by deleting a perfect matching, has a game chromatic number of $n$, but the game chromatic number drops to $2$ if a single isolated vertex is added to the graph \cite{ZhuTuza}. Therefore, many upper bounds for the game chromatic number of certain graph classes, such as the bounds for planar graphs and graphs of bounded treewidth given above, are obtained by studying the graph marking game. 

We will consider the relationship between the game chromatic number of a graph $G$ and the properties of the bicolored subgraphs of $G$ with respect to some fixed proper coloring. In particular, we will show that for a graph $G$ with a proper coloring, the game chromatic number of $G$ is bounded above by a function of the number of colors used to color $G$ and the game coloring numbers of the bicolored subgraphs of $G$. Our method will generalize the method of Dinski and Zhu used to prove the inequality $\chi_g(G) \leq \chi_a(G)(\chi_a(G) + 1)$, as the method of Dinski and Zhu essentially just uses the fact that each bicolored subgraph of an acyclically colored graph $G$ is a forest, which must have a small game coloring number. One corollary of our method will be that the Cartesian product of two graphs of bounded game coloring number must have a bounded game chromatic number, which answers a question of Zhu from \cite{ZhuCartesian}.

The paper will be organized as follows. In Section \ref{secMarking}, we prove that a properly colored graph whose bicolored subgraphs have bounded game coloring number must have a bounded game chromatic number, and we list a number of corollaries. Then, in Section \ref{secProducts}, we apply the method of Section \ref{secMarking} to calculate upper bounds on the game chromatic numbers of certain graph products, namely the Cartesian product and the strong product of two graphs. Finally, in Section \ref{secCon}, we pose some questions.

\section{Bounding $\chi_g$ with the game coloring number of bicolored subgraphs}
\label{secMarking}
In this section, we will show that the game chromatic number of a properly colored graph $G$ may be bounded by a function of the number of colors used to color $G$ and the game coloring numbers of the bicolored subgraphs of $G$. Dinski and Zhu show in \cite{DinskiZhu} that for a graph $G$, $\chi_g(G) \leq \chi_a(G)(\chi_a(G) + 1)$, where $\chi_a(G)$ is the acyclic chromatic number of $G$. We will follow the ideas of Dinski and Zhu to prove a more general upper bound on $\chi_g(G)$ in terms of the game coloring numbers of the bicolored subgraphs of $G$ with respect to some proper coloring.

We consider a slight variation of the graph marking game, which we name the \emph{Bob marking game}. In the Bob marking game on a graph $G$, Alice and Bob play on $G$ by the same rules as those in the graph marking game, but Alice marks with a red pen, and Bob marks with a blue pen. In the Bob marking game, we let Bob move first. When a play of the game is finished, for each vertex $v \in V(G)$, we define the score of $v$ as the number of neighbors of $v$ marked in blue at the time $v$ was marked. In other words, only the neighbors of $v$ marked by Bob contribute to the score of $v$. Then, for a play of the Bob marking game on $G$, we say that the score of the play is equal to the maximum score over all vertices of $V(G)$, plus one. We say that the value $\Bob(G)$ is equal to the minimum integer $t$ for which Alice has a strategy to limit the score of a play of the Bob marking game on $G$ to $t$. Defining $\col_g^B(G)$ to be the lowest score achievable by Alice in the traditional marking game on $G$ with optimal play when Bob moves first, it is clear that $\Bob(G) \leq \col_g^B(G)$. Furthermore, Zhu remarks in \cite{ZhuCartesian} that $\col_g^B(G) \leq \col_g(G) + 1$, so it follows that $\Bob(G) \leq \col_g(G)+1$.

It is worth giving an example of a graph $G$ for which $\Bob(G) < \col_g^B(G)$ in order to show that these parameters are indeed different. It is shown in \cite{Bodlaender} that there exists forests $F$ for which $\chi_g(F) = \col_g^B(F) = 4$. In contrast, we will prove that $\Bob(F) \leq 3$ holds for every forest $F$ using the following strategy for Alice, which is used implicitly by Dinski and Zhu in \cite{DinskiZhu}. At a given state of the Bob marking game on $F$, let $F'$ denote the subgraph of $F$ that is obtained by removing from $F$ the vertices marked in red by Alice, as well as the edges whose endpoints are both marked in blue by Bob. We will show that at the end of each of Alice's turns, she can ensure that at most one vertex from each component of $F'$ is marked in blue by Bob. Alice can certainly ensure that this condition holds at the end of her first turn. Now, suppose that the condition holds at the end of Alice's $i$th turn. On Bob's $(i+1)$th turn, Bob chooses a component $K$ of $F'$ and marks a vertex $v \in K$ blue. If $v$ is the only blue vertex in $K$, then Alice marks an arbitrary vertex, and the condition is satisfied again at the end of Alice's $(i+1)$th turn. Otherwise, there exists a single other blue vertex $w \in K$. If $w$ is a neighbor of $v$, then the edge $vw$ is removed from $F'$. Then, Alice marks an arbitrary vertex, and the condition holds again at the end of Alice's $(i+1)$th turn. On the other hand, if $w$ is not a neighbor of $v$, then there exists a unique path $P$ in $K$ connecting $v$ and $w$ with at least one internal vertex. Alice marks an internal vertex $u$ of $P$, and then since $u$ is removed from $F'$, the condition again holds at the end of Alice's $(i+1)$th turn. Now, if $\Bob(F) \geq 4$, then at some point in the game, an unmarked vertex $u$ must have three blue neighbors, and $u$ along with these three blue neighbors belong to a single component of $F'$. However, Alice's strategy ensures that at any point in the game, a component of $F'$ contains at most two blue vertices, giving us a contradiction. Therefore, $\Bob(F) \leq 3$.

Now, consider a graph $G$ for which $\Bob(G) \leq t$. In a play of the Bob marking game on $G$, Alice has a strategy in which every vertex $v \in V(G)$ is marked before the number of blue marked neighbors of $v$ exceeds $t-1$. We say that Alice's strategy on $G$ with respect to the bound $\Bob(G) \leq t$ is \emph{reactive} if for each vertex $v$, if $v$ ever has $t-1$ blue marked neighbors after Bob's move, then Alice marks $v$ immediately. 
For example, the strategy for forests $F$ described above is reactive with respect to the bound $\Bob(F) \leq 3$, because if Bob ever marks two neighbors of an unmarked vertex $u$, then Alice will immediately mark $u$. 
If Alice plays a strategy on $G$ to limit the score of each vertex to $t-1$, then the only way that Alice's strategy would not be reactive would be if Alice were to allow a vertex to remain unmarked when all of its neighbors were marked, with exactly $t-1$ neighbors marked in blue. Indeed, if an unmarked vertex $v$ has $t-1$ blue marked neighbors and at least one unmarked neighbor on Bob's turn, then Bob can achieve a score of $t+1$ on $G$ by marking an additional neighbor of $v$, so in any successful strategy, Alice would need to mark $v$ to prevent its score from increasing. Most strategies that we consider for a graph $G$ that give a bound of the form $\Bob(G) \leq t$ will be reactive, as it is not usually convenient to try to ensure that all neighbors of an unmarked vertex $v$ are marked, and it is usually easier for Alice just to mark a vertex $v$ in order to prevent its score from increasing.

The following theorem generalizes the method of Dinski and Zhu originally used to prove that for any graph $G$, $\chi_g(G) \leq \chi_a(G)(\chi_a(G) + 1)$ \cite{DinskiZhu}. The method of Dinski and Zhu considers an acyclically colored graph $G$, and using the acyclical coloring of $G$, these authors devise a winning strategy for Alice in the graph coloring game on $G$. Using the strategy above, Dinski and Zhu implicitly show that $\Bob(F) \leq 3$ holds for every forest $F$, and they essentially use the fact that every bicolored subgraph $H$ of $G$ satisfies $\Bob(H) \leq 3$ to devise their strategy. The following theorem, however, shows that in order to bound the game chromatic number of a properly colored graph $G$, it is enough simply to ensure that $\Bob(H)$ is bounded for every bicolored subgraph of $G$. We use the term \emph{$k$-coloring} to refer to a proper graph coloring using $k$ colors.

\begin{theorem}
\label{thmMark}
Let $G$ be a graph with a $k$-coloring $\phi$, and suppose that every bicolored subgraph $H$ of $G$ with respect to $\phi$ satisfies $\Bob(H) \leq t$. If Alice has a reactive strategy with respect to each graph $H$ and the bound $\Bob(H) \leq t$, then
$$\chi_g(G) \leq k((k-1)(t-2) + 2).$$
\end{theorem}
\begin{proof}
Let $\phi$ be a proper coloring of $G$ using $k$ colors that satisfies the assumptions of the theorem. In order to show that $\chi_g(G) \leq  k((k-1)(t-2) + 2)$, we must show that Alice has a winning strategy in the graph coloring game using $k((k-1)(t-2)+2)$ colors. We will define a set $C$ of $k((k-1)(t-2)+2)$ values with which Alice and Bob will play the graph coloring game, and to avoid confusion, we will refer to the values in $C$ as \emph{shades}, rather than colors. That is, on each turn, we will let Alice or Bob assign a shade from $C$ to a vertex of $G$ that has not already been assigned a shade. On the other hand, we will refer to the $k$ values in the image of $\phi$ as \emph{colors}. We will partition $C$ into $k$ parts of size $(k-1)(t-2)+2$, and we will say that for each color $c$ used by $\phi$, $C$ contains $(k-1)(t-2)+2$ shades of $c$.

For two colors $c$ and $d$, let $G_{c,d} \subseteq G$ be the subgraph of $G$ induced by the vertices of $V(G)$ that are colored with $c$ and $d$ by $\phi$. Let $S_{c,d}$ be a reactive strategy of the marking game on $G_{c,d}$ by which Alice can limit the score of any vertex of $G_{c,d}$ to $t-1$ in the Bob marking game. We will describe Alice's strategy for the coloring game on $G$. In Alice's strategy, Alice will always color some vertex $v \in V(G)$ with a shade of $\phi(v)$. We will sometimes allow Alice to choose an arbitrary vertex $v$ to assign a shade of $\phi(v)$, and in this case, we say that Alice plays an \emph{idle move}. 

As Alice plays the game, Alice will in fact consider $k \choose 2$ different Bob marking games played on the graphs $G_{c,d}$, one for each color pair $c,d \in \phi(V(G))$. Each time Bob makes a move, Alice will consider Bob's move to be a move in a Bob marking game on one of the subgraphs $G_{c,d}$. Alice will calculate a response to Bob's move in the Bob marking game on $G_{c,d}$ using the strategy $S_{c,d}$, and based on Alice's response in the Bob marking game on $G_{c,d}$, Alice will respond to Bob's move in the coloring game on $G$.

Alice's strategy is as follows. Alice begins the game with an idle move. On each of Bob's turns, if Bob chooses a vertex $v \in V(G)$ and colors $v$ with a shade of $\phi(v)$, then Alice responds by playing an idle move. If Bob colors a vertex $v$ with a shade $c$ that is not one of the shades of $\phi(v)$, then Alice considers Bob's move as if it were a move in the Bob marking game on $G_{c,\phi(v)}$. Alice then uses $S_{c,\phi(v)}$ to choose a vertex $w \in V(G)$ to mark in response to Bob's move in the Bob marking game on $G_{c,\phi(v)}$. Then, in the coloring game on $G$, Alice colors $w$ with any available shade of $\phi(w)$. If $w$ has already been colored, then Alice plays an idle move. Alice repeats this process for each of Bob's moves.

We now show that Alice's strategy always succeeds in producing a proper coloring of $G$. Suppose that on some turn, Alice attempts to color a vertex $v$ with a shade of $\phi(v)$. For any neighbor $w$ of $v$ that is colored with a shade of $\phi(v)$, $w$ must have been colored by Bob. Equivalently, $w$ must have been marked by Bob in the Bob marking game on $G_{\phi(v),\phi(w)}$. However, Alice has used the strategy $S_{\phi(v), \phi(w)}$ to ensure that Bob does not mark more than $t-1$ neighbors of an unmarked vertex in the Bob marking game on $G_{\phi(v), \phi(w)}$. Therefore, for each color $c \in \phi(V(G))$ that appears in the neighborhood of $v$, at most $t-1$ vertices $w \in N(v)$ with $\phi(w) = c$ have been colored by Bob with a shade of $\phi(v)$. Furthermore, as the strategy $S_{\phi(v),\phi(w)}$ is reactive, there exists at most one color $c^*$ for which $t-1$ vertices $w \in N(v)$ with $\phi(w) = c^*$ have been colored with a shade of $\phi(v)$, and this color $c^*$ must satisfy $c^* = \phi(w^*)$, where $w^* \in N(v)$ is the vertex that has just been colored by Bob with a shade of $\phi(v)$ on the last move. For all other colors $c \in \phi(V(G))$, at most $t-2$ neighbors $w \in N(v)$ with $\phi(w) = c$ have been colored by Bob with a shade of $\phi(v)$. This implies that the total number of shades of $\phi(v)$ that appear in the neighborhood of $v$ is at most $(t-2)(k-1)+1$. As Alice has $(t-2)(k-1)+2$ shades of $\phi(v)$ to use, Alice thus has an available shade of $\phi(v)$ to use at $v$. Hence, Alice's strategy succeeds at every move.

As Alice always succeeds in coloring a vertex of $G$ with a shade from $C$ on her turn, the only way that $G$ would not be properly colored would be if Bob were unable to color any vertex of $G$ on some turn, in which case the coloring game would end prematurely with Alice losing. However, Bob may always ``pretend to be Alice" and successfully color a vertex of $V(G)$ with an idle move using the previous argument. Therefore, Bob also always has a legal move, and hence $G$ is properly colored.
\end{proof}
We make several observations about Theorem \ref{thmMark} and its proof. First, we have defined the graph coloring game with Alice moving first, but it is easy to see that the strategy of Theorem \ref{thmMark} works regardless of which player moves first. Second, the upper bound on $\chi_g(G)$ from Theorem \ref{thmMark} also holds for any subgraph of $G$, as removing edges from $G$ does not make the strategy any more difficult for Alice, and if a vertex of $G$ that Alice wishes to color is not present in some subgraph, then Alice may simply play an idle move. Third, while the Bob marking game is not a standard part of the literature, the inequality $\Bob(H) \leq \col^B_g(H) \leq \col_g(H)+1$ implies that we can replace the condition $\Bob(H) \leq t$ of Theorem \ref{thmMark} with a bound using more standard parameters. Finally, if $\Bob(H) \leq t$ holds for every bicolored subgraph $H$ of $G$, but Alice does not necessarily have a reactive strategy with respect to these bounds, then a very similar argument gives the following upper bound, which is only slightly worse than the bound in Theorem \ref{thmMark}.
\begin{corollary}
\label{corNonreactive}
Let $G$ be a graph with a $k$-coloring $\phi$, and suppose that every two-colored subgraph $H$ of $G$ with respect to $\phi$ satisfies $\Bob(H) \leq t$. Then
$$\chi_g(G) \leq k((k-1)(t-1) + 1).$$
\end{corollary}

We note that the strategy of Zhu in \cite{ZhuCartesian} used to bound the game chromatic number of graph Cartesian products bears some resemblance to the strategy of Theorem \ref{thmMark}, as in \cite{ZhuCartesian}, Zhu explicitly devises a single graph coloring strategy by combining many graph marking strategies on smaller subgraphs. However, the strategy of Zhu in \cite{ZhuCartesian} still relies on acyclic colorings, so the strategy of Theorem \ref{thmMark} is the first strategy, to the best of our knowledge, that uses more general bicolored subgraphs. 

Theorem \ref{thmMark} has a number of corollaries. First, the upper bound of Dinski and Zhu from \cite{DinskiZhu} follows immediately.
\begin{corollary}
\label{corAcy}
For every graph $G$, $\chi_g(G) \leq \chi_a(G)(\chi_a(G) +1)$.
\end{corollary}
\begin{proof}
We have shown previously that $\Bob(F) \leq 3$ holds for every forest $F$, and furthermore, that Alice has a strategy that is reactive with respect to this bound. In an acyclic coloring on $G$, every bicolored subgraph on $G$ is a forest, so letting $k = \chi_a(G)$ and $t = 3$ in Theorem \ref{thmMark} yields the result.
\end{proof}
Additionally, a number of similar upper bounds follow.
\begin{corollary}
Let $G$ be a graph with a proper $k$-coloring in which every bicolored subgraph has treewidth at most $w$. Then $\chi_g(G) \leq k(3w(k-1)+2)$. 
\end{corollary}
\begin{proof}
Zhu shows in \cite{ZhuKtrees} that for a graph $H$ of treewidth at most $w$, $\col^B_g(H) \leq 3w + 2$, and furthermore, the strategy for Alice that Zhu gives is reactive with respect to this bound. Therefore, letting $t = 3w + 2$ in Theorem \ref{thmMark} yields the result.
\end{proof}
\begin{corollary}
Let $G$ be a graph with a proper $k$-coloring in which every bicolored subgraph is planar. Then $\chi_g(G) \leq k(15k - 13)$. 
\end{corollary}
\begin{proof}
Zhu shows in \cite{ZhuRefined} that for a planar graph $H$, $\col^B_g(H) \leq 17$, and furthermore, the strategy for Alice that Zhu gives is reactive with respect to this bound. Therefore, letting $t = 17$ in Theorem \ref{thmMark} yields the result.
\end{proof}
\begin{corollary}
Let $G$ be a graph with a proper $k$-coloring in which every bicolored subgraph is of genus at most $g$. Then $\chi_g(G) \leq k((k-1)\lfloor \frac{1}{2}(3 + \sqrt{1 + 48g} + 19) \rfloor+2)$. 
\end{corollary}
\begin{proof}
Zhu shows in \cite{ZhuKtrees} that for a graph $H$ of genus at most $g$, $\col^B_g(H) \leq \lfloor \frac{1}{2}(3 + \sqrt{1 + 48g} + 23) \rfloor$, and furthermore, the strategy for Alice that Zhu gives is reactive with respect to this bound. Therefore, letting $t = \lfloor \frac{1}{2}(3 + \sqrt{1 + 48g} + 23) \rfloor$ in Theorem \ref{thmMark} yields the result.
\end{proof}

It is natural to ask whether these upper bounds for the game chromatic number of a graph obtained using the method of Theorem \ref{thmMark} are optimal. Giving an overall answer to this question is difficult, as graph colorings in which bicolored subgraphs have bounded game coloring number have not yet received any attention. It is known, however, that Corollary \ref{corAcy} often does not give a tight upper bound. For instance, using the fact that a planar graph has an acyclic chromatic number of at most $5$ \cite{Borodin}, Corollary \ref{corAcy} implies that a planar graph has a game chromatic number of at most $30$, a result shown in \cite{DinskiZhu}, but as stated, a different method of Zhu shows that the game chromatic number of a planar graph is in fact at most $17$ \cite{ZhuRefined}.

\section{The Cartesian product and strong product of graphs}
\label{secProducts}
In this section, we will show that Theorem \ref{thmMark} may be used to calculate an upper bound on certain graph products, namely the Cartesian product of two graphs and the strong product of two graphs. We first consider the Cartesian products of two graphs, which we define as follows. Given two graphs $G_1$ and $G_2$, the \emph{Cartesian product} of $G_1$ and $G_2$, written $G_1 \square G_2$, is defined as the graph on the vertex set $V(G_1) \times V(G_2)$ in which two vertices $(u,v$) and $(u',v')$ are adjacent if and only if either $u = u'$ and $v\sim v'$ in $G_2$, or $v = v'$ and $u\sim u'$ in $G_1$, where $\sim$ represents adjacency. An example of the Cartesian product of two graphs is shown in Figure \ref{figCartEx}. In \cite{ZhuCartesian}, Zhu calculates an upper bound on the game chromatic number of the Cartesian product $G_1 \square G_2$ of two graphs $G_1$ and $G_2$, but Zhu's upper bound relies on the acyclic chromatic number of one of the graphs and the game coloring number of a modified form of the other graph. Using Theorem \ref{thmMark}, however, we may show that $\chi_g(G_1 \square G_2)$ may be bounded above only by $\col_g(G_1)$ and $\col_g(G_2)$. Recall that for a graph $G$, we define $\col_g^B(G)$ to be the lowest score achievable by Alice in the graph marking game on $G$ with optimal play when Bob moves first.

\begin{figure}
\begin{center}
\begin{tikzpicture}
[scale=2,auto=left,every node/.style={circle,fill=gray!30},minimum size = 6pt,inner sep=1pt]
\draw (0,0)--(2,0);
\draw (0,0)--(0,1);

\node (p1) at (0.25,-0.25) [draw = black]  {};
\node (p2) at (1,-0.25) [draw = black]  {};
\node (p3) at (1.75,-0.25) [draw = black]  {};

\node (t1) at (-0.25,0.25) [draw = black]  {};
\node (t2) at (-0.25,0.625) [draw = black]  {};
\node (t3) at (-0.125,0.875) [draw = black]  {};

\node (t11) at (0.25,0.25) [draw = black]  {};
\node (t12) at (0.25,0.625) [draw = black]  {};
\node (t13) at (0.375,0.875) [draw = black]  {};

\node (t21) at (1,0.25) [draw = black]  {};
\node (t22) at (1,0.625) [draw = black]  {};
\node (t23) at (1.125,0.875) [draw = black]  {};

\node (t31) at (1.75,0.25) [draw = black]  {};
\node (t32) at (1.75,0.625) [draw = black]  {};
\node (t33) at (1.875,0.875) [draw = black]  {};

\foreach \from/\to in {t1/t2,t2/t3,t1/t3,p1/p2,p2/p3,t11/t12,t11/t13,t12/t13, t21/t22,t21/t23,t22/t23, t31/t32,t31/t33,t32/t33,t11/t21,t21/t31,t12/t22,t22/t32,t13/t23,t23/t33}
    \draw (\from) -- (\to);
\end{tikzpicture}
\end{center}
\caption{The figure shows a $K_3$, a path of length $2$, and the Cartesian product of these two graphs.}
\label{figCartEx}
\end{figure}
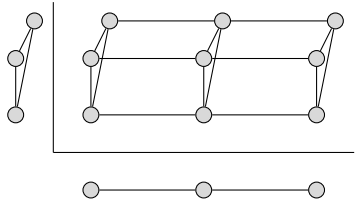

\begin{theorem}
\label{thmCartesian}
Let $G_1$ and $G_2$ be graphs. Let $k = \chi(G_1)\chi(G_2)$, and let $t = \max\{\col^B_g(G_1),\col^B_g(G_2)\}$. Then $$\chi_g(G_1 \square G_2) \leq k((k-1)(t-1) + 1).$$
\end{theorem}
\begin{proof}
Let $\phi_1:E(G_1) \rightarrow \{1,2,\dots,\chi(G_1)\}$ be a proper coloring of $G_1$, and let $\phi_2:E(G_2) \rightarrow \{1,2,\dots,\chi(G_2)\}$ be a proper coloring of $G_2$. For each pair $v_1 \in V(G_1)$, $v_2 \in V(G_2)$, we may color the corresponding vertex $(v_1, v_2) \in V(G_1 \square G_2)$ with the color $(\phi_1(v_1),\phi_2(v_2))$, which gives us a proper coloring 
$$\phi:E(G_1 \square G_2) \rightarrow \{1,2,\dots,\chi(G_1)\} \times \{1,2,\dots,\chi(G_2)\}$$
 using $k$ colors.

We claim that each connected bicolored subgraph $H$ of $G_1 \square G_2$ under $\phi$ satisfies $\col_g^B(H) \leq t$. Indeed, let $H \subseteq G_1 \square G_2$ be a connected bicolored subgraph with respect to $\phi$. If $H$ is a single vertex, then $\col^B_g(H) = 1$; otherwise, $H$ has at least one edge $e$. We assume without loss of generality that $e$ has endpoints $(u,v_1), (u,v_2)$, where $u \in V(G_1)$, and $v_1, v_2 \in V(G_2)$, and hence that $H$ is colored with the colors $(\phi_1(u),\phi_2(v_1))$ and $(\phi_1(u),\phi_2(v_2))$. If every vertex of $H$ is of the form $(u,v)$ for some $v \in V(G_2)$, then $H$ is isomorphic to a subgraph of $G_2$, and hence $\col_g^B(H) \leq t$. Otherwise, as $H$ is connected, $H$ must contain a vertex of the form $(u',v)$, where $u' \in V(G_1)$ is a neighbor of $u$ in $G_1$, and $v \in V(G_2)$ is any vertex in $G_2$. However, as $u$ and $u'$ are neighbors, $\phi_1(u) \neq \phi_1(u')$, so $\phi(u',v)$ cannot be one of $(\phi_1(u),\phi_2(v_1))$ and $(\phi_1(u),\phi_2(v_2))$, a contradiction to the assumption that $H$ is bicolored. Therefore, $H$ is isomorphic to a subgraph of $G_2$, and $\Bob(H) \leq \col_g^B(H) \leq t$. The same upper bound holds even if $H$ is not connected, as $\col_g^B(H)$ is equal to the maximum value $\col_g^B(H')$ over all components $H'$ of $H$. As Alice does not necessarily have a reactive strategy with respect to the game coloring numbers of $G_1$ and $G_2$, we apply Corollary \ref{corNonreactive} with our values $k$ and $t$, and we obtain an upper bound of $\chi_g(G_1 \square G_2) \leq k((k-1)(t-1) + 1)$, which completes the proof.
\end{proof}

We note that given $r$ graphs $G_1, \dots, G_r$, we may use the same method to obtain the upper bound $\chi_g(G_1 \square \cdots \square G_r) \leq k((k-1)(t-1)+1)$, where $k = \chi(G_1) \cdots \chi(G_r)$, and $t = \max\{\col^B_g(G_1),\dots, \col^B_g(G_r)\}$.

\begin{corollary}
\label{corBd}
Let $G_1, G_2$ be graphs, and let $t = \max\{\col_g(G_1),\col_g(G_2)\}$. Then $$\chi_g(G_1 \square G_2) \leq t^2((t^2-1)t + 1) = t^5  -t^3 + t^2.$$
\end{corollary}
\begin{proof}
The bound follows directly from Theorem \ref{thmCartesian} after applying the inequalities $\chi(G_i) \leq \chi_g(G_i) \leq \col_g(G_i) \leq t$ and $\col_g^B(G_i) \leq t+1$ for $i = 1,2$.
\end{proof}

Corollary \ref{corBd} answers a question of Zhu from \cite{ZhuCartesian} asking if $\chi_g(G_1 \square G_2)$ is bounded whenever $\col_g(G_1)$ and $\col_g(G_2)$ are bounded. Zhu asks this question for the graph coloring game in which Bob moves first, but the original strategy from Theorem \ref{thmMark} works the same regardless of which player moves first.
The upper bounds of Theorem \ref{thmCartesian} and Corollary \ref{corBd} are often far from tight, however. For example, Theorem \ref{thmCartesian} tells us that the game chromatic number of the Cartesian product of two planar graphs is at most $16(15\cdot 16 + 1) = 3856$, but using a different method, Zhu \cite{ZhuCartesian} obtains a sharper upper bound of $105$. Furthermore, in the following example, we show two graphs $G_1$ and $G_2$ for which the Cartesian product $G_1 \square G_2$ has a game chromatic number equal to the trivial lower bound of $\chi(G_1 \square G_2) = \max\{\chi(G_1),\chi(G_2)\}$, which is far from the upper bound given in Theorem \ref{thmCartesian}.

For an even integer $n \geq 2$, let $G_1$ be the union of a complete graph $K_n$ and a single isolated vertex, and let $G_2$ be the union of an edge $K_2$ and a single isolated vertex. We illustrate $G_1$, $G_2$, and their Cartesian product in Figure \ref{figCartesian}. $G_1 \square G_2$ has four components: a $K_n$ component, a $K_n \square K_2$ component, a single vertex component, and a $K_2$ component. We observe that $\chi(G_1 \square G_2) = n$, and we will show that $\chi_g(G_1 \square G_2) = n$ by giving a strategy using $n$ colors with which Alice may win the graph coloring game on $G_1 \square G_2$. In comparison, the upper bound for $\chi_g(G_1 \square G_2)$ given by Theorem \ref{thmCartesian} is $4n^3 -2n^2 + 2n$, which is far from optimal.

Alice's strategy is as follows. On the first move, Alice colors the isolated vertex of $G_1 \square G_2$ with any color. Then, whenever Bob colors a vertex in a component $C$ of $G_1 \square G_2$, Alice colors a vertex of $C$ on the next move. As each component of $G_1 \square G_2$ of size at least $2$ has an even number of vertices, Alice will always be able to respond to Bob by coloring a vertex in the same component that Bob just colored, provided that each uncolored vertex still has a legal color. Therefore, in order to show that $\chi_g(G_1 \square G_2) = n$, it suffices to show that Alice wins the coloring game with $n$ colors on each component of $G_1 \square G_2$ of size at least $2$ when Bob moves first.

It is clear that Alice wins the coloring game on $K_2$ and $K_n$ with $n$ colors when Bob moves first; thus, we will only explicitly describe Alice's strategy for winning the coloring game on $K_n \square K_2$. Let $K_n \square K_2$ have $2n$ vertices $u_0, \dots, u_{n-1}, v_0, \dots, v_{n-1}$, so that $u_i \sim u_j$ and $v_i \sim v_j$ for each pair $0 \leq i < j \leq n-1$, and so that $u_i \sim v_i$ for each $0 \leq i \leq n-1$. Alice will play as follows. Whenever Bob colors a vertex $u_i$ with a color $c$, Alice will respond by coloring $v_{i+1}$ with $c$, and whenever Bob colors a vertex $v_i$ with a color $c$, Alice will respond by coloring $u_{i-1}$ with $c$, with addition calculated modulo $n$. It is easy to check that after each of Alice's turns, the partial coloring on $u_0, \dots, u_{n-1}$ is equal to the partial coloring at $v_0, \dots, v_{n-1}$, but ``shifted down" by one. Therefore, Alice's strategy always gives her a legal move, and together Alice and Bob will complete a proper coloring of $K_n \square K_2$ using $n$ colors. Therefore, $\chi_g(G_1 \square G_2) = \chi(G_1 \square G_2) = n$, which is much smaller than the upper bound we would obtain from Theorem \ref{thmCartesian}. 

\begin{figure}
\begin{center}
\begin{tikzpicture}
[scale=2,auto=left,every node/.style={circle,fill=gray!30},minimum size = 6pt,inner sep=1pt]
\node (z) at (0.875,-0.25) [fill = white]  {$K_n$};
\node (z) at (0.875,0.375) [fill = white]  {$K_n \square K_2$};
\node (z) at (0.875,0.875) [fill = white]  {$K_n$};
\node (z) at (1,-0.625) [fill = white]  {$G_1$};
\node (z) at (-0.625,0.5) [fill = white]  {$G_2$};
\draw (0,0)--(2,0);
\draw (0,0)--(0,1);
\draw (0.25,-0.125)--(1.5,-0.125);
\draw (0.25,-0.375)--(1.5,-0.375);
\draw (0.25,-0.375)--(0.25,-0.125);
\draw (1.5,-0.125)--(1.5,-0.375);

\draw (0.25,0.125)--(1.5,0.125);
\draw (0.25,0.625)--(1.5,0.625);
\draw (0.25,0.125)--(0.25,0.625);
\draw (1.5,0.125)--(1.5,0.625);

\draw (0.25,.75)--(1.5,.75);
\draw (0.25,1)--(1.5,1);
\draw (0.25,1)--(0.25,.75);
\draw (1.5,.75)--(1.5,1);

\node (z) at (1.75,-0.25) [draw = black]  {};
\node (a) at (-0.25,0.25) [draw = black]  {};
\node (b) at (-0.25,0.5) [draw = black]  {};
\node (z) at (-0.25,0.875) [draw = black]  {};
\node (z) at (1.75,0.875) [draw = black]  {};

\node (a2) at (1.75,0.25) [draw = black]  {};
\node (b2) at (1.75,0.5) [draw = black]  {};

\foreach \from/\to in {a/b,a2/b2}
    \draw (\from) -- (\to);
\end{tikzpicture}
\end{center}
\caption{The figure shows two graphs $G_1$ and $G_2$ along with their Cartesian product $G_1 \square G_2$. In this example, $\chi_g(G_1 \square G_2) = \chi(G_1 \square G_2)$, showing that the upper bound in Theorem \ref{thmCartesian} may be far from optimal.}
\label{figCartesian}
\end{figure}
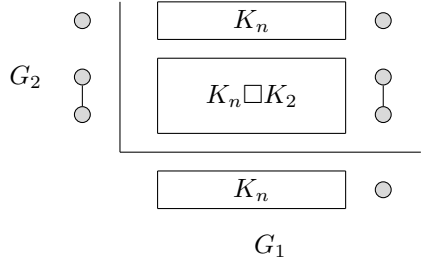

Next, Theorem \ref{thmMark} allows us to establish the following result about the strong product of two graphs. Given two graphs $G_1$ and $G_2$, the \emph{strong product} of $G_1$ and $G_2$, written $G_1 \boxtimes G_2$, is defined as the graph on the vertex set $V(G_1) \times V(G_2)$ in which two vertices $(u,v)$ and $(u',v')$ are adjacent if and only if both of the following hold:
\begin{itemize}
\item $u = u'$, or $u \sim u'$ in $G_1$;
\item $v = v'$, or $v \sim v'$ in $G_2$.
\end{itemize}
An example of the strong product of two graphs is illustrated in Figure \ref{figStrongEx}. Furthermore, given a graph $G$, the \emph{square} of $G$, written $G^2$, is defined as the graph on $V(G)$ in which two distinct vertices $u,v \in V(G)$ are adjacent in $G^2$ if and only if $u$ and $v$ are at a distance of at most $2$ in $G$. With these definitions in place, we have the following result.

\begin{figure}
\begin{center}
\begin{tikzpicture}
[scale=2,auto=left,every node/.style={circle,fill=gray!30},minimum size = 6pt,inner sep=1pt]
\draw (0,0)--(2,0);
\draw (0,0)--(0,1);

\node (p1) at (0.25,-0.25) [draw = black]  {};
\node (p2) at (1,-0.25) [draw = black]  {};
\node (p3) at (1.75,-0.25) [draw = black]  {};

\node (t1) at (-0.25,0.25) [draw = black]  {};
\node (t2) at (-0.25,0.625) [draw = black]  {};
\node (t3) at (-0.125,0.875) [draw = black]  {};

\node (t11) at (0.25,0.25) [draw = black]  {};
\node (t12) at (0.25,0.625) [draw = black]  {};
\node (t13) at (0.375,0.875) [draw = black]  {};

\node (t21) at (1,0.25) [draw = black]  {};
\node (t22) at (1,0.625) [draw = black]  {};
\node (t23) at (1.125,0.875) [draw = black]  {};

\node (t31) at (1.75,0.25) [draw = black]  {};
\node (t32) at (1.75,0.625) [draw = black]  {};
\node (t33) at (1.875,0.875) [draw = black]  {};

\foreach \from/\to in {t1/t2,t2/t3,t1/t3,p1/p2,p2/p3,t11/t12,t11/t13,t12/t13, t21/t22,t21/t23,t22/t23, t31/t32,t31/t33,t32/t33,t11/t21,t21/t31,t12/t22,t22/t32,t13/t23,t23/t33,t11/t22,t11/t23,t12/t21,t12/t23,t13/t21,t13/t22,t21/t32,t21/t33,t22/t31,t22/t33,t23/t31,t23/t32}
    \draw (\from) -- (\to);
\end{tikzpicture}
\end{center}
\caption{The figure shows a $K_3$, a path of length $2$, and the strong product of these two graphs.}
\label{figStrongEx}
\end{figure}
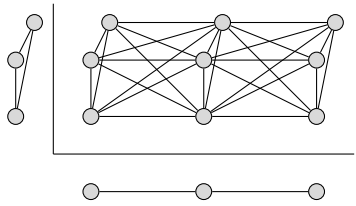

\begin{theorem}
\label{thmStrong}
Let $G_1$ and $G_2$ be graphs, let $t = \col_g(G_1)$, and let $k = \chi (G_1) \chi(G_2^2)$. Then $$\chi_g(G_1 \boxtimes G_2) \leq k((k-1)t + 1).$$
\end{theorem}
\begin{proof}
Let $\phi_1:V(G_1) \rightarrow \{1,2, \dots, \chi(G_1)\}$ be a proper coloring of $G_1$, and let $\phi_2:V(G_2) \rightarrow \{1,2,\dots,\chi(G^2_2)\}$ be a proper coloring of $G^2_2$. As in Theorem \ref{thmCartesian}, we define a proper coloring 
$$\phi:E(G_1 \boxtimes G_2) \rightarrow \{1,2, \dots, \chi(G_1)\} \times \{1,2,\dots, \chi(G_2^2)\}$$
using $k$ colors by coloring each vertex $(u,v) \in G_1 \boxtimes G_2$ such that $\phi(u,v) = (\phi_1(u), \phi_2(v))$.

If $G_1$ is an independent set, then $G_1 \boxtimes G_2$ consists of copies of $G_2$, so $\chi_g(G_1 \boxtimes G_2) \leq \Delta(G_2) + 1  \leq \chi(G_2^2)$, as $G_2^2$ has a clique of size $\Delta(G_2)+1$. Hence, the theorem holds in this case, and we thus assume that $G_1$ has at least one edge, and hence that $t \geq 2$.

Consider a connected bicolored subgraph $H$ of $G_1 \boxtimes G_2$ with respect to $\phi$. We aim to show that $\col_g(H) \leq t$. If $H$ contains no edge, then $\col_g(H) = 1$. If $H$ contains an edge of the form $(u_1,v)(u_2,v)$ for vertices $u_1,u_2 \in V(G_1)$ and $v \in V(G_2)$, then by the argument of Theorem \ref{thmCartesian}, $H$ is isomorphic to a subgraph of $G_1$, and hence $\col_g(H) \leq t$. Similarly, if $H$ contains an edge of the form $(u,v_1)(u,v_2)$ for vertices $u \in V(G_1)$ and $v_1,v_2 \in V(G_2)$, then by the argument of Theorem \ref{thmCartesian}, $H$ is isomorphic to a subgraph of $G_2$. However, as $\phi_2$ is a proper coloring of $G_2^2$, $v_2$ is the only neighbor of $v_1$ in $G_2$ with color $\phi_2(v_2)$, and $v_1$ is the only neighbor of $v_2$ in $G_2$ with color $\phi_2(v_1)$. Hence, $H$ must be isomorphic to $K_2$, and $\col_g(H) = 2 \leq t$. 

Finally, suppose $H$ contains an edge of the form $(u_1,v_1)(u_2,v_2)$ for two adjacent vertices $u_1,u_2 \in V(G_1)$ and two adjacent vertices $v_1,v_2 \in V(G_2)$. Again, as $\phi_2$ is a proper coloring of $G_2^2$, $v_1$ and $v_2$ must be the only vertices of $G_2$ that appear as the second entry in an element of $V(H)$. Furthermore, as $\phi_1$ is a proper coloring of $G_1$, every edge of $H$ must be of the form $(u,v_1)(u',v_2)$, where $u,u' \in V(G_1)$ may be any distinct pair of adjacent vertices in $G_1$. 
We recall that $H$ is colored with two colors and hence that $H$ is bipartite. Therefore, $H$ is isomorphic to the subgraph $G' \subseteq G_1$ induced by the vertices $u \in V(G_1)$ that appear in some pair $(u,v_i) \in V(H)$, where $i \in \{1,2\}$, and we see that the index $i$ of the pair $(u,v_i)$ in which a vertex $u \in V(G')$ appears indicates to which partite set of $G'$ the vertex $u$ belongs. Hence, $\col_g(H) \leq t$, and furthermore, $\col_g^B(H) \leq t + 1$.


In each case, the bound $\col^B_g(H) \leq t+1$ holds even when $H$ is not connected, as the value of $\col^B_g(H)$ is equal to the maximum value $\col^B_g(H')$ over all components $H'$ of $H$, and $\col^B_g(H') \leq \col_g(H') + 1 \leq t + 1$. Hence, we have a proper coloring $\phi$ of $G$ using $k$ colors in which $\col^B_g(H) \leq t  + 1$ holds for each bicolored subgraph $H$ of $G$. Then, the result follows from Corollary \ref{corNonreactive}.
\end{proof}
Theorem \ref{thmStrong} has the following corollary, which shows that a the strong product of a graph with bounded game coloring number and a second graph of bounded degree must have bounded game chromatic number. 
\begin{corollary}
\label{corStrong}
Let $G$ be a graph, and let $G'$ be a graph of maximum degree $\Delta$. Then 
$$\chi_g(G \boxtimes G') \leq \chi(G)^2 (\Delta^2+1)^2 \col_g(G) \leq (\Delta^2+1)^2 \col_g(G)^3.$$
\end{corollary}
\begin{proof}
The chromatic number of the square of $G'$ is at most $\Delta^2+1$, so the result follows from Theorem \ref{thmStrong} by letting $t = \col_g(G)$, using the fact that $k \leq \chi(G)(\Delta^2+1)$, and noting that the upper bound of Theorem \ref{thmStrong} is at most $k^2t$.
\end{proof}
Corollary \ref{corStrong} tells us, for instance, that the strong product of any graph $G$ with a cubic graph has a game chromatic number of at most $100\col_g(G)^3$, and that the strong product of a planar graph $G$ with a graph of maximum degree $\Delta$ has a game chromatic number of at most $4^2 \cdot 17 (\Delta^2+1)^2 = 272 (\Delta^2+1)^2$. However, Theorem \ref{thmStrong} and Corollary \ref{corStrong} are likely far from best possible. Furthermore, if we consider two complete graphs $K_m$ and $K_n$, we see that $\chi_g(K_m \boxtimes K_n) = \chi(K_m \boxtimes K_n) = mn$, so it is possible for the strong product $G_1 \boxtimes G_2$ of two graphs $G_1$ and $G_2$ to have a game chromatic number equal to the trivial lower bound of $\chi(G_1 \boxtimes G_2)$, which is far from the upper bound of Theorem \ref{thmStrong}.
\section{Conclusion}
\label{secCon}
We have shown in Corollary \ref{corBd} that if two graphs $G_1$ and $G_2$ have their game coloring numbers bounded by a constant, then $\chi_g(G_1 \square G_2)$ is also bounded by a constant.
It seems natural to try to strengthen this result by asking whether $\col_g(G_1 \square G_2)$ is also bounded by a constant; however, Bartnicki et al.~\cite{Bartnicki} have shown $\col_g(G_1 \square G_2)$ is unbounded when $G_1 = G_2 = K_{1,n}$, while $\col_g(K_{1,n}) \leq 4$.

On the other hand, we have shown in Corollary \ref{corStrong} that given a graph $G_1$ of bounded game coloring number and a graph $G_2$ of bounded degree, $\chi_g(G_1 \boxtimes G_2)$ is bounded by a constant. However, the following question remains open, which could strengthen Corollary \ref{corBd} and Corollary \ref{corStrong}.
\begin{question}
Let $G_1$ and $G_2$ be graphs, and suppose that $\col_g(G_1)$ and $\col_g(G_2)$ are both bounded by a constant. Is it true that $\chi_g(G_1 \boxtimes G_2)$ is bounded by a constant?
\end{question}

\section{Acknowledgment}
I am grateful to Bojan Mohar for his helpful advice regarding the organization and and presentation of the results in this manuscript, and for pointing out an error in an earlier version of this manuscript. I am also grateful to the referees for their helpful comments.

\raggedright
\bibliographystyle{plain}
\bibliography{gameColoring}

\end{document}